\documentclass[12pt,final]{amsart}
\usepackage{amsmath}%
\usepackage{amsfonts}%
\usepackage{amssymb}%
\usepackage{graphicx}
\usepackage{pictexwd,dcpic}
\usepackage{mathrsfs}
\usepackage{booktabs}
\usepackage{amsthm}
\usepackage{floatrow}
\usepackage{caption}
\usepackage{subcaption}

\theoremstyle{plain}
\newtheorem{definition}[equation]{Definition}
\newtheorem{theorem}[equation]{Theorem}

\newtheorem{lemma}[equation]{Lemma}

\newtheorem{corollary}[equation]{Corollary}
\newtheorem{remark}[equation]{Remark}

\theoremstyle{definition}
\newtheorem{example}[equation]{Example}

\DeclareMathOperator{\Vol}{Vol}
\DeclareMathOperator{\Int}{int}

\DeclareMathOperator{\spanv}{span}

\newcommand{\rom}[1]{\uppercase\expandafter{\romannumeral #1\relax}}

\makeatother

\newcommand{\defv}[1]{\textbf{\textit{#1}}}

 {\begin{list}{}%
         {\setlength{\leftmargin}{#1}}%
         \item[]%
 }
 {\end{list}}

\begin{document}
\bibliographystyle{alpha}

\title[]{On the densest packing of polycylinders in any dimension}
\author[W\"oden Kusner]{W\"oden Kusner\\Department of Mathematics\\ University of Pittsburgh}
\thanks{Research partially supported by NSF grant 1104102.}
\maketitle

\begin{abstract}
Using transversality and a dimension reduction argument, a result of A. Bezdek and W. Kuperberg is applied to polycylinders $\mathbb{D}^2\times \mathbb{R}^n$, showing that the optimal packing density is $\pi/\sqrt{12}$ in any dimension. 
\end{abstract}



\section{Introduction}
G. Fejes T\`oth and W. Kuperberg  \cite{toth1993blichfeldt} describe a method for computing an upper bound for packings by infinite polycylinders  -- objects isometric to $\mathbb{D}^2 \times \mathbb{R}^n$ in $\mathbb{R}^{n+2}$.  This article explicitly computes their bound and then proves the sharp upper bound of $\pi/\sqrt{12}$ for the packing density of infinite polycylinders in any dimension, generalizing a result of A. Bezdek and W. Kuperberg \cite{bezdek1990maximum}.  

\subsection{Objects considered}
Open and closed Euclidean unit $n$-balls will be denoted $\mathbb{B}^n$ and $\mathbb{D}^n$ respectively. The closed unit interval is denoted $\mathbb{I}$. A general polycylinder $C$ is a set isometric to $\Pi_{i=1}^{i=m}\lambda_i\mathbb{D}^{k_i}$ in $\mathbb{R}^{ k_1+\dots + k_m}$, where $\lambda_i$ is in $[0,\infty]$.  For this article, the term polycylinder refers to the special case of an infinite polycylinder over a two-dimensional disk of unit radius.

\begin{definition}
A \defv{polycylinder} is a set isometric to\, $\mathbb{D}^2 \times \mathbb{R}^n $ in $ \mathbb{R}^{n+2}$.  
\end{definition}

The following are standard definitions.

\begin{definition}
A \defv{polycylinder packing of\, $\mathbb{R}^{n+2}$} is a countable family $$\mathscr{C} = \{C_i\}_{i \in I}$$ of polycylinders $C_i \subset \mathbb{R}^{n+2}$ with mutually disjoint interiors.
\end{definition}

\begin{definition}
The \defv{upper density} $\delta^+ (\mathscr{C})$ of a packing $\mathscr{C}$ of\, $\mathbb{R}^n$ is defined to be 
 \[ \delta^+ (\mathscr{C}) = \limsup_{r\rightarrow \infty}\frac{ \Vol(\mathscr{C}\cap r\mathbb{B}^n)}{\Vol(r\mathbb{B}^n)}.\]
  Note that this notion can be generalized further by replacing $\mathbb{B}^n$ with an arbitrary convex body $K.$
 \end{definition}
 \begin{definition}
 The \defv{upper packing density}\, $\delta^+ (C)$ of an object\, $C$ is the supremum of $\delta^+ (\mathscr{C})$ over all packings $\mathscr{C}$ of\, $\mathbb{R}^n$ by \,$C$.
\end{definition}

The definition of density is equivalent to a number of other definitions under some mild assumptions.  For more on density, see for example \cite{brass2005research, radin2004orbits, conway1999recent, federer1969geometric}.

\section{Computing a bound of Fejes T\'oth and Kuperberg }\label{comp}
 In \cite{toth1993blichfeldt}, Fejes T\'oth and Kuperberg describe a method for computing upper bounds on a broad class of objects in any dimension: Blichfeldt-type results for balls \cite{blichfeldt1929minimum, rankin1947closest} extend to results for outer parallel bodies.  Fejes T\'oth and Kuperberg compute upper bounds for the density of packings by cylinders $\mathbb{D}^{n-1} \times t\mathbb{I}$ and also for packings by outer parallel bodies of line segments, i.e. capped cylinders, in $\mathbb{R}^n$ but do not address the case of polycylinders. This section recalls that work and explicitly computes the \cite{toth1993blichfeldt} bound to be $\delta^+(\mathscr{C}) \le 0.941533\dots$ for polycylinder packings in any dimension.  
 
\subsection{Background}
Given a packing of $\mathbb{R}^n$ by congruent objects $\mathscr{C} = \{C_i\}_{i\in I}$, there are a fixed body $C \subset \mathbb{R}^n$ and isometries $\{\phi_i\}_{i\in I}$ of $\mathbb{R}^n$ such that $C_i = \phi_iC$ for all $i$ in $I$.


\begin{definition}
 A function $f:\mathbb{R}^{n} \rightarrow \mathbb{R}^+$ is a \defv{Blichfeldt gauge} for a convex body $C\subset \mathbb{R}^n$ if for any collection of isometries $\Phi=\{\phi_i\}_{i\in I}$ of $\mathbb{R}^n$ where $\mathscr{C} = \{\phi_iC\}_{i\in I}$ is a packing and for all $x$ in $\mathbb{R}^n,$
 $$\sigma_\Phi(f)(x) := \sum_{i\in I} f(\phi_i^{-1}x) \le 1.$$
\end{definition}

Notice that the characteristic function $\bold{1}_C$ of $C$ is a Blichfeldt gauge for $C$. Replacing $\bold{1}_C$ with a more general Blichfeldt gauge $f$ lets one replace the characteristic function of the packing $\bold{1}_{\mathscr{C}}$ with a diffuse version $\sigma_\Phi(f)$. This new function $\sigma_\Phi(f)$ has the same general characteristics as $\bold{1}_{\mathscr{C}}$, is still bounded pointwise by $1$ in the ambient space and is uniformly bounded independent of $\Phi$ in the moduli space of packings. As $f$ may have greater mass than $\bold{1}_C$, this allows one to estimate the volume of the interstices of a packing and thereby bound the packing density.

\begin{example}\label{exmp} Blichfeldt initially uses the radial function $2f_0$ where

\begin{displaymath}
   f_0(r) = \left\{
     \begin{array}{ll}
       \frac{1}{2}(2-r^2) & : 0 \le r \le \sqrt{2} \\
       0 & : r > \sqrt{2}
     \end{array}
   \right.
\end{displaymath}  
and showed that $f_0$ is a Blichfeldt gauge. Then, for a packing $\mathscr{C} =\{\phi_i C\}_{i\in I}$ of a cube $t\mathbb{I}^n$ by spheres, the support of $\sigma_\Phi(f)$ is contained in a slightly larger cube $(t+2\sqrt{2}-2)\mathbb{I}^n$. A bound on sphere packing density can then be extracted as follows. From the definition of the Blichfeldt gauge and integrating in spherical coordinates, one finds

$$(t+2\sqrt{2}-2)^n \ge \left\vert I \right\vert\int_{\mathbb{R}^n}\hspace*{-7pt} f_0\, \mathrm{dV} = \frac{\left\vert I \right\vert\Vol(\mathbb{B}^n)2^{\frac{n+2}{2}}}{n+2}.$$

When density is measured relative to a cube,

$$\delta^+_{}(\mathscr{C}) = \frac{\left\vert I \right\vert \Vol(\mathbb{B}^n)}{t^n} \le \frac{n+2}{2^{\frac{n+2}{2}}}\left(1+\frac{2\sqrt{2}-2}{t}\right)^n \hspace*{-7pt}.$$

It is easy to see that the same method works when $t\mathbb{I}^n$ is replaced with $\mathbb{B}^n_{t/2}$.  By passing to the limit, the bound $$\delta^+(\mathscr{C}) \le  \frac{n+2}{2^{\frac{n+2}{2}}}$$ holds for any sphere packing in $\mathbb{R}^n.$ 
\end{example}

\subsection{Blichfeldt-type bound for polycylinders}


Example \ref{exmp} motivates the following general observations.

\begin{theorem}[Blichfeldt]\label{blich} If $g$ is a Blichfeldt gauge for a body $C$, then $\delta^+(\mathscr{C})\le \Vol(C)/J(g)$ where 
$$J(g) = \int_{\mathbb{R}^n}\hspace*{-7pt}g\, \mathrm{dV}.$$
\end{theorem}

\begin{theorem}[Fejes T\'oth--Kuperberg]
If $f(\alpha)$, $\alpha \ge 0$, is a real valued function such that $f(|x|)$ is a Blichfeldt gauge for the unit ball, and $C$ is a convex body with inradius $r(C)$, then for any $\varrho \le r(C)$
$$g(x) = f\left(\frac{d(x,C_{-\varrho})}{\varrho}\right)$$
is a Blichfeldt gauge for $C$, where $C_{-\varrho}$ is the inner parallel body of $C$ at distance $\varrho$.
\end{theorem}

For their more general results, Fejes T\'oth and Kuperberg do not use $f_0$, but rather Blichfeldt's modified version.
\begin{definition} The \defv{modified Blichfeldt gauge} \cite{blichfeldt1929minimum} for\, $\mathbb{D}^n$ is the radial function
\begin{displaymath}
   f_1(r) = \left\{
     \begin{array}{ll}
       1 & : 0 \le r \le 2-\sqrt{2} \\
       \frac{1}{2}(2-r)^2 & : 2 - \sqrt{2}\le r \le 1\\
        \frac{1}{2}(2-r^2) & : 1\le r\le \sqrt{2}\\
       0 & : r > \sqrt{2}
     \end{array}
   \right.
\end{displaymath}  
\end{definition}

\begin{definition}[Fejes T\'oth--Kuperberg]For the two-dimensional gauge $f_1$ defined above, $$A_2 := J(f_1)/\Vol(\mathbb{D}^2) = (29-16\sqrt{2})/6.$$
\end{definition}

From the previous theorems and definitions, the results of \cite{toth1993blichfeldt} give a rough estimate for the density of infinite polycylinders as follows.  Consider $C(t) = \mathbb{D}^{n+2} + t\mathbb{I}^n$ in $\mathbb{R}^{n+2}$ and the gauge $g_t(x) =f_1(d(x,C(t)_{-1}))$, where $f_1$ is the modified Blichfeldt gauge and $C(t)_{-1}$ is the inner parallel body at distance $1$, i.e. an $n$-cube of height $t$. From Theorem \ref{blich} an estimate of the integral$ \int_{\mathbb{R}^{n+2}}g_t \, \mathrm{dV}$ gives a density bound. By integrating $g_t$ over $C(t)_{-1} \times \mathbb{R}^2$ and noticing that contribution from the complement $\mathbb{R}^{n+2} \smallsetminus (C(t)_{-1} \times \mathbb{R}^2)$ are of strictly lower order -- it is bounded above by a constant times the $(n-1)$-Hausdorff measure of the boundary $\partial C(t)_{-1} \subset C(t)_{-1}$, it follows that
 $$\delta^+(C(t)) \le \frac{\pi t^n}{\pi A_2 t^n + O(t^{n-1})}.$$
In the limit as $t$ goes to infinity, this gives a bound of $1/A_2 = .941533...$ for infinite polycylinders in any dimension.

\begin{remark}
The Blichfeldt gauge method yields constrained optimization problems of the form
$$\textrm{maximize } J(g)
\textrm{ subject to }\sigma_\Phi(g) \le 1 \textrm{ for all }\Phi \textrm{ where } \Phi(C) \textrm{ is a packing.}$$
\end{remark}

\section{Transversality}\label{trans}
This section introduces the required transversality arguments in affine geometry.
\begin{definition}
A \defv{d-flat} is a d-dimensional affine subspace of\, $\mathbb{R}^n$.
\end{definition}

\begin{definition}
The \defv{parallel dimension} $dim_\parallel \{F,\dots, G\}$ of a collection of flats $\{F, \dots, G\}$ is the dimension of their maximal parallel sub-flats.
\end{definition} 

The notion of parallel dimension can be interpreted in several ways, allowing a modest abuse of notation.  

For a collection of flats $\{F, \dots, G\}$, consider their tangent cones at infinity $\{F_\infty, \dots, G_\infty\}$.  The parallel dimension of $\{F, \dots, G\}$ is the dimension of the intersection of these tangent cones.  This may be viewed as the limit of a rescaling process $\mathbb{R}^n \rightarrow  r\mathbb{R}^n$ as $r$ tends to $0$, leaving only the scale-invariant information.

For a collection of flats $\{F, \dots, G\}$, consider each flat as a system of linear equations. The corresponding homogeneous equations determine a collection of linear subspaces $\{F_\infty, \dots, G_\infty\}$. The parallel dimension is the dimension of their intersection $F_\infty \,\cap \dots \cap\, G_\infty$. 

\begin{definition}
Two disjoint $d$-flats are \defv{parallel} if their parallel dimension is $d$, that is, if every line in one is parallel to a line in the other.
\end{definition}

\begin{definition}
Two disjoint d-flats are \defv{skew} if their parallel dimension is less than $d.$
\end{definition}

\begin{figure}[t]
\centering
   \begin{subfigure}{0.4\linewidth} \centering
     \includegraphics[scale=0.5]{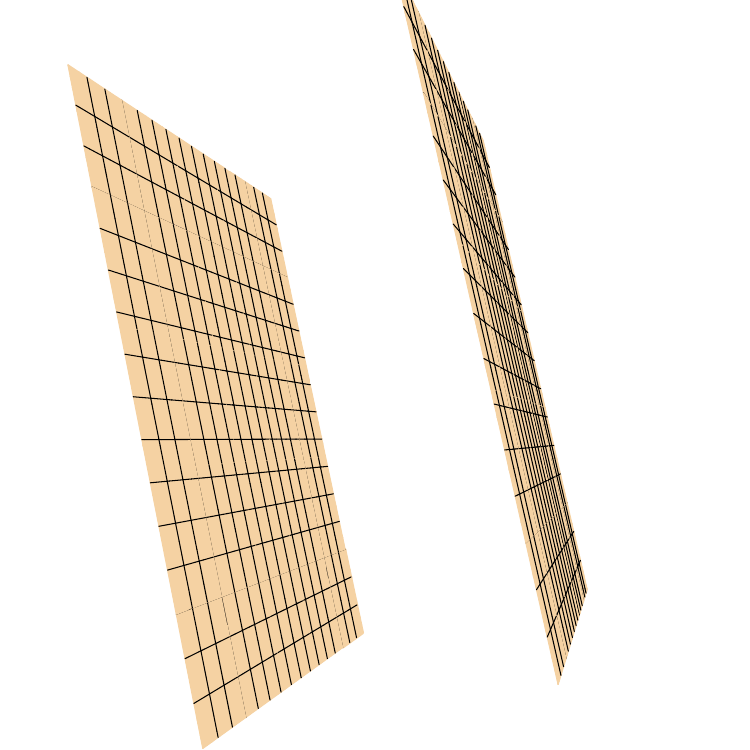}
     \subcaption*{(a)}
   \end{subfigure}
   \begin{subfigure}{0.4\linewidth} \centering
     \includegraphics[scale=0.5]{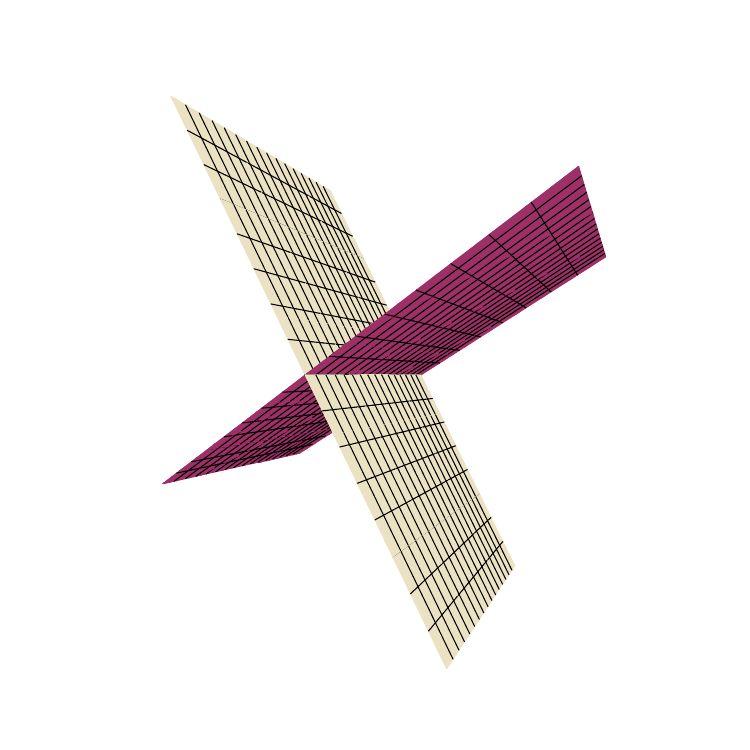}
      \subcaption*{(b)}
       \end{subfigure}
       \caption{\newline Disjoint 2-flats in $\mathbb{R}^4$ with (a) parallel dimension 2 and (b) parallel dimension 1.} \label{fig:twofigs}
\end{figure}

\begin{lemma}
A pair of disjoint $n$-flats in\, $\mathbb{R}^{n+k}$ with $n\ge k$ has parallel dimension strictly greater than $n-k.$
\end{lemma}

\begin{proof}
 By homogeneity of $\mathbb{R}^{n+k}$, let $F=F_\infty.$  As $F_\infty$ and $G$ are disjoint, $G$ contains a non-trivial vector $\bold{v}$ such that $G = G_\infty +\bold{v}$ and $\bold{v}$ is not in $F_\infty + G_\infty.$  It follows that
$$dim (\mathbb{R}^{n+k}) \ge dim(F_\infty + G_\infty + \spanv(\bold{v}) )> dim(F_\infty + G_\infty) $$
$$= dim (F_\infty) + dim (G_\infty) - dim (F_\infty \cap G_\infty).$$
Count dimensions to find $n+k > n +n - dim_\parallel (F_\infty,G_\infty).$
\end{proof}
\begin{corollary}\label{dimcor}
A pair of disjoint $n$-flats in\, $\mathbb{R}^{n+2}$ has parallel dimension at least $n-1$.
\end{corollary}




\section{Results}

\subsection{Pairwise foliations and dimension reduction}



\begin{figure}[t]
\centering
   \begin{subfigure}{0.3\linewidth} \centering
     \includegraphics[scale=0.3]{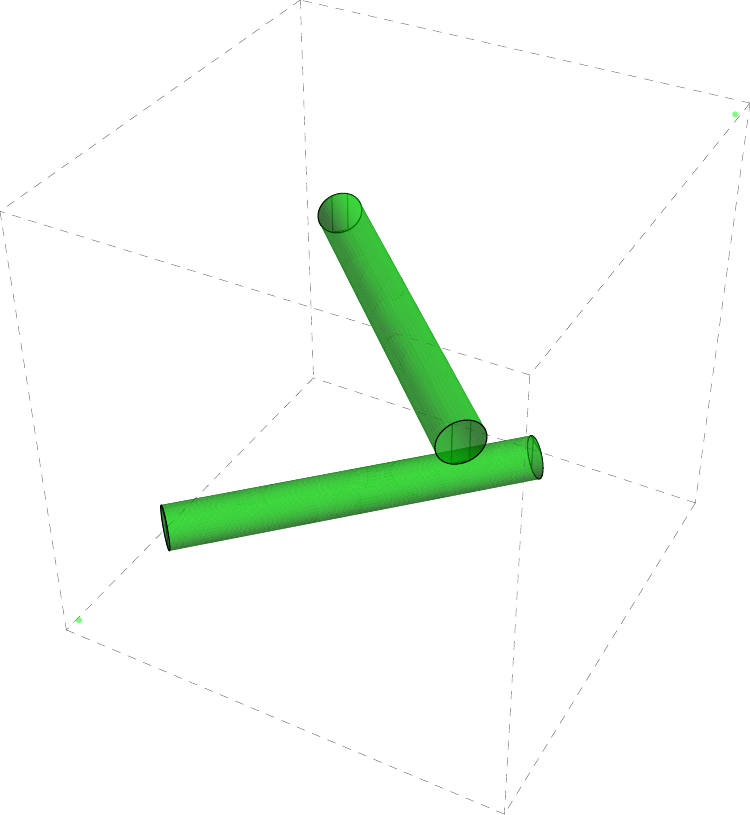}
     \subcaption*{$t=-\epsilon$}
   \end{subfigure}
   \begin{subfigure}{0.3\linewidth} \centering
     \includegraphics[scale=0.3]{Fig2.pdf}
      \subcaption*{$t=0$}
       \end{subfigure}
        \begin{subfigure}{0.3\linewidth} \centering
     \includegraphics[scale=0.3]{Fig2.pdf}
      \subcaption*{$t=\epsilon$}
       \end{subfigure}
\caption{\newline The leaves of a pair of polycylinders in $\mathbb{R}^4$ indexed by the common parallel direction $t$ are identical.} \label{fig:twofigs}
\end{figure}

\begin{definition}
The \defv{core} $a_i$ of a polycylinder $C_i$ isometric to $\mathbb{D}^2~\times~\mathbb{R}^n$ in $\mathbb{R}^{n+2}$ 
is the distinguished $n$-flat defining $C_i$ as the set of points at most distance $1$ from $a_i$. 
\end{definition}

In a packing $\mathscr{C}$ of\, $\mathbb{R}^{n+2}$ by polycylinders, Corollary \ref{dimcor} shows that, for every pair of polycylinders $C_i$ and $C_j$, one can choose parallel $(n-1)$-dimensional subflats $b_i \subset a_i$ and $b_j \subset a_j$ and define a product foliation $$\mathscr{F}^{b_i,b_j}:\mathbb{R}^{n+2} \rightarrow \mathbb{R}^{n-1} \times \mathbb{R}^3$$ with $\mathbb{R}^3$ leaves that are orthogonal to $b_i$ and to $b_j$. Given a point $x$ in $a_i$, there is a distinguished $\mathbb{R}^3$ leaf $F_x^{b_i,b_j}$ that contains the point $x$. The foliation $\mathscr{F}^{b_i,b_j}$ restricts to foliations of $C_i$ and $C_j$ with right-circular-cylinder leaves. 


\subsection{The Dirichlet slice}

\begin{definition}
In a packing $\mathscr{C}$ \hspace*{-5pt} of\, $\mathbb{R}^{n+2}$ by polycylinders, the \defv{Dirichlet cell} $D_i$ associated to a polycylinder $C_i$ is the set of points in $\mathbb{R}^{n+2}$ no further from $C_i$ than from any other polycylinder in $\mathscr{C}$.  
\end{definition}
The Dirichlet cells of a packing partition $\mathbb{R}^{n+2},$ because $C_i \subset D_i$ for all polycylinders $C_i$.  To bound the density $\delta^+(\mathscr{C})$, it is enough to fix an $i$ in $I$ and consider the density of $C_i$ in $D_i$. For the Dirichlet cell $D_i,$ there is a slicing as follows.
\begin{definition}
Given a fixed a polycylinder $C_i$ in a packing $\mathscr{C}$ \hspace*{-7pt} of\, $\mathbb{R}^{n+2}$ by polycylinders and a point $x$ on the core $a_i$, the plane $p_x$ is the $2$-flat orthogonal to $a_i$ and containing the point $x$. The \defv{Dirichlet slice} $d_x$ is the intersection of $D_i$ and $p_x.$
\end{definition}

Note that $p_x$ is a sub-flat of $F_x^{b_i,b_j}$ for all $j$ in $I.$

\subsection{Bezdek--Kuperberg bound}

For any point $x$ on the core $a_i$ of a polycylinder $C_i$, the results of Bezdek and Kuperberg \cite{bezdek1990maximum} apply to the Dirichlet slice $d_x$.

\begin{lemma}\label{lem1}
A Dirichlet slice is convex and, if bounded, a parabola-sided polygon.
\end{lemma}
\begin{proof}

Construct the Dirichlet slice $d_x$ as an intersection.  Define $d^j$ to be the set of points in $p_x$ no further from $C_i$ than from $C_j$.  Then the Dirichlet slice $d_x$ is realized as 
$$d_x = \{\cap_{j\in I} d^j \}.$$ 

Each arc of the boundary of $d_x$ in $p_x$ is given by an arc of the boundary of some $d^j$ in $p_x.$ The boundary of $d^j$ in $p_x$ is the set of points in $p_x$ equidistant from $C_i$ and $C_j.$ Since the foliation $\mathscr{F}^{b_i,b_j}$  is a product foliation, the arc of the boundary of $d^j$ in $p_x$ is also the set of points in $p_x$ equidistant from the leaf $C_i \cap F_x^{b_i,b_j}$ of $\mathscr{F}^{b_i,b_j}|_{C_i}$ and the leaf $C_j \cap F_x^{b_i,b_j}$ of $\mathscr{F}^{b_i,b_j}|_{C_j}$. This reduces the analysis to the case of a pair cylinders in $\mathbb{R}^3.$ From \cite{bezdek1990maximum}, it follows that $d^j$ is convex and the boundary of $d_j$ in $p_x$ is a parabola; the intersection of such sets $d^j$ in $p_x$ is convex, and a parabola-sided polygon if bounded.
 \end{proof}
 

 Let $S_x(r)$ be the circle of radius $r$ in $p_x$ centered at $x$. 
\begin{lemma} \label{lem2}The vertices of $d_x$ are not closer to $S_x(1)$ than the vertices of a regular hexagon circumscribed about $S_x(1).$
\end{lemma}
\begin{proof}
A vertex of $d_x$ occurs where three or more polycylinders are equidistant, so the vertex is the center of a $(n+1)$-ball $B$ tangent to three polycylinders.  Thus $B$ is tangent to three disjoint unit $(n+2)$-balls $B_1$, $B_2$, $B_3$.  By projecting into the affine hull of the centers of $B_1$, $B_2$, $B_3$, it is immediate that the radius of $B$ is no less than $2/\sqrt{3} -1.$
\end{proof}

\begin{lemma} \label{lem3}
Let $y$ and $z$ be points on the circle $S_x(2/\sqrt{3})$.  If each of $y$ and $z$ is equidistant from $C_i$ and $C_j$, then the angle $yxz$ is smaller than or equal to $2\arccos (\sqrt{3} -1) = 85.8828\dots^\circ.$ 
\end{lemma}

\begin{proof} Following \cite{bezdek1990maximum, kusner2013upper}, the existence of a supporting hyperplane of $C_i$ that separates $\Int(C_i)$ from $\Int(C_j)$ suffices.
\end{proof}



In \cite{bezdek1990maximum}, it is shown that planar objects satisfying Lemmas \ref{lem1}, \ref{lem2} and \ref{lem3} have area no less than $\sqrt{12}.$ As the bound holds for all Dirichlet slices, it follows that $\delta^+(\mathbb{D}^2\times \mathbb{R}^n) \le \pi/\sqrt{12}$ in $\mathbb{R}^{n+2}.$  The product of the dense disk packing in the plane with $\mathbb{R}^n$ gives a polycylinder packing in $\mathbb{R}^{n+2}$ that achieves this density.  Combining this with the result of Thue \cite{thue1910densest} for $n=0$ and the result of Bezdek and Kuperberg \cite{bezdek1990maximum} for $n=1$, it follows that

\begin{theorem}
$\delta^+(\mathbb{D}^2 \times \mathbb{R}^n) \le \pi/\sqrt{12} $ for all natural numbers $n.$
\end{theorem}

\section{Acknowledgments}

Thanks to T. Hales, W. Kuperberg and R. Kusner. The author was supported by Austrian Science Fund (FWF) Project 5503 and National Science Foundation (NSF) Grant No. 1104102.

\bibliography{ThomTom}
\end{document}